\documentclass[11pt,reqno,letterpaper]{amsart}
\usepackage{amssymb}
\usepackage{graphicx,color}
\usepackage{hyperref}
\usepackage[normalem]{ulem}
\usepackage{fullpage} 

\usepackage{amssymb}
\usepackage{graphicx,color}
\usepackage{hyperref}
\usepackage[normalem]{ulem}

\usepackage[utf8]{inputenc}
\usepackage{gensymb}
\usepackage[square,numbers]{natbib}
\usepackage{amssymb}
\usepackage{amsfonts}
\usepackage{latexsym}
\usepackage{amsmath,systeme}
\usepackage{graphics}
\usepackage{graphicx}
\usepackage{epsfig}
\usepackage[leqno]{mathtools}
\usepackage{amsthm}
\usepackage[dvipsnames,usenames]{xcolor}

\usepackage{soul}

\usepackage{fullpage} 

\mathtoolsset{showonlyrefs}

\def\vint_#1{\mathchoice%
          {\mathop{\kern 0.2em\vrule width 0.6em height 0.69678ex depth -0.58065ex
                  \kern -0.8em \intop}\nolimits_{\kern -0.4em#1}}%
          {\mathop{\kern 0.1em\vrule width 0.5em height 0.69678ex depth -0.60387ex
                  \kern -0.6em \intop}\nolimits_{#1}}%
          {\mathop{\kern 0.1em\vrule width 0.5em height 0.69678ex
              depth -0.60387ex
                  \kern -0.6em \intop}\nolimits_{#1}}%
          {\mathop{\kern 0.1em\vrule width 0.5em height 0.69678ex depth -0.60387ex
                  \kern -0.6em \intop}\nolimits_{#1}}}
\def\vintslides_#1{\mathchoice%
          {\mathop{\kern 0.1em\vrule width 0.5em height 0.697ex depth -0.581ex
                  \kern -0.6em \intop}\nolimits_{\kern -0.4em#1}}%
          {\mathop{\kern 0.1em\vrule width 0.3em height 0.697ex depth -0.604ex
                  \kern -0.4em \intop}\nolimits_{#1}}%
          {\mathop{\kern 0.1em\vrule width 0.3em height 0.697ex depth -0.604ex
                  \kern -0.4em \intop}\nolimits_{#1}}%
          {\mathop{\kern 0.1em\vrule width 0.3em height 0.697ex depth -0.604ex
                  \kern -0.4em \intop}\nolimits_{#1}}}

\newcommand{\kint}{\vint}

\newcommand{\aveint}[2]{\mathchoice%
          {\mathop{\kern 0.2em\vrule width 0.6em height 0.69678ex depth -0.58065ex
                  \kern -0.8em \intop}\nolimits_{\kern -0.45em#1}^{#2}}%
          {\mathop{\kern 0.1em\vrule width 0.5em height 0.69678ex depth -0.60387ex
                  \kern -0.6em \intop}\nolimits_{#1}^{#2}}%
          {\mathop{\kern 0.1em\vrule width 0.5em height 0.69678ex depth -0.60387ex
                  \kern -0.6em \intop}\nolimits_{#1}^{#2}}%
          {\mathop{\kern 0.1em\vrule width 0.5em height 0.69678ex depth -0.60387ex
                  \kern -0.6em \intop}\nolimits_{#1}^{#2}}}

\usepackage{scalerel}[2014/03/10]
\usepackage[usestackEOL]{stackengine}

\def\dashint{\,\ThisStyle{\ensurestackMath{%
  \stackinset{c}{.2\LMpt}{c}{.5\LMpt}{\SavedStyle-}{\SavedStyle\phantom{\int}}}%
  \setbox0=\hbox{$\SavedStyle\int\,$}\kern-\wd0}\int}
\def\ddashint{\,\ThisStyle{\ensurestackMath{%
  \stackinset{c}{.2\LMpt}{c}{.5\LMpt+.2\LMex}{\SavedStyle-}{%
    \stackinset{c}{.2\LMpt}{c}{.5\LMpt-.2\LMex}{\SavedStyle-}{%
      \SavedStyle\phantom{\int}}}}\setbox0=\hbox{$\SavedStyle\int\,$}\kern-\wd0}\int}

\theoremstyle{plain}

\def \A{\mathcal A}

\DeclareMathOperator{\tr}{trace}
\def \eps{\varepsilon}
\def \R{\mathbb R}

\begin{document}


\theoremstyle{plain}
\newtheorem{theorem}{Theorem} [section]
\newtheorem{corollary}[theorem]{Corollary}
\newtheorem{lemma}[theorem]{Lemma}
\newtheorem{proposition}[theorem]{Proposition}
\newtheorem{example}[theorem]{Example}


\theoremstyle{definition}
\newtheorem{definition}[theorem]{Definition}
\theoremstyle{remark}
\newtheorem{remark}[theorem]{Remark}

\numberwithin{theorem}{section}
\numberwithin{equation}{section}

\title[Asymptotic Mean value formulas for parabolic nonlinear equations]{Asymptotic Mean value formulas for parabolic nonlinear equations}

\thanks{P.B. partially supported by the Academy of Finland project no. 298641. \\
\indent F.C. partially supported by MICINN grants MTM2017-84214-C2-1-P and PID2019-110712GB-I100 (Spain).
 \\
\indent J.D.R. partially supported by 
CONICET grant PIP GI No 11220150100036CO
(Argentina), PICT-2018-03183 (Argentina) and UBACyT grant 20020160100155BA (Argentina).
}

\author[P. Blanc]{Pablo Blanc}
\address{{\sc Pablo Blanc.} Department of Mathematics and Statistics, University of Jyv\"askyl\"a, PO Box 35, FI-40014 Jyv\"askyl\"a, Finland.}
\email{pblanc@dm.uba.ar}

\author[F. Charro]{Fernando Charro}
\address{{\sc Fernando Charro.} Department of Mathematics, Wayne State University, 656 W. Kirby, Detroit, MI 48202, USA.}
\email{fcharro@wayne.edu}
\thanks{}

\author[J. J. Manfredi]{Juan J. Manfredi}
\address{{\sc Juan J. Manfredi.}
Department of Mathematics,
University of Pittsburgh, Pittsburgh, PA 15260, USA.
}
\email{manfredi@pitt.edu}

\author[J. D. Rossi]{Julio D. Rossi}
\address{{\sc Julio D. Rossi.} Departamento  de Matem\'atica, FCEyN, Universidad de Buenos Aires,
 Pabell\'on I, Ciudad Universitaria (1428),
Buenos Aires, Argentina.}
\email{jrossi@dm.uba.ar}

\newcommand{\pablo}[1]{{\color{blue}{#1}}}

\date{}

\begin{abstract} In this paper we characterize viscosity solutions to nonlinear parabolic
equations (including parabolic Monge-Amp\`ere equations) by 
asymptotic mean value formulas.
Our asymptotic mean value formulas can be interpreted from a probabilistic point of view   in terms of Dynamic Programming
Principles for certain two-player, zero-sum games.
\end{abstract}

\maketitle

\section{Introduction} \label{sec.intro}

\subsection{Asymptotic mean value formulas for elliptic equations}
It is a classical well-known fact that  a function $u$ is harmonic 
($u$ is a solution to $\Delta u= \mbox{trace}(D^2u)=0$) in a domain
$\Omega\subset\mathbb{R}^n$ 
 if and only if $u$ satisfies the mean value property
\begin{equation}\label{MVPrperty_HarmonicFunctions}
u(x)
=
\dashint_{B_\varepsilon(x)}u(y)\, dy
\end{equation}
for each $x\in\Omega$ and all $0 < \varepsilon< \textrm{dist}(x,\partial\Omega)$. 
In fact, a weaker statement, an asymptotic version of the mean value property, suffices to characterize  harmonic functions. A continuous function $u$ is harmonic in $\Omega$ if and only if 
\begin{equation}\label{Asymptotic.MVP.Harmonic}
u(x)=\dashint_{B_\varepsilon(x)} u(y)\,dy+o(\varepsilon^2)\quad\textrm{as}\ \varepsilon\to0,
\end{equation}
see \cite{Blaschke,Ku,Privaloff}.
Moreover, the mean value property can be used to characterize sub- and superharmonic functions replacing the equality by the appropriate inequality
in \eqref{MVPrperty_HarmonicFunctions} and \eqref{Asymptotic.MVP.Harmonic}.
A discrete version of the asymptotic mean value property also holds; a continuous function $u$ is harmonic if and only if  
\begin{equation}\label{Asymptotic.MVP.Harmonic.discrete}
u(x)= \frac1n \sum_{j=1}^n \left\{ \frac12 u(x+\varepsilon e_j) + \frac12 u(x-\varepsilon e_j) \right\}
+o(\varepsilon^2)\quad\textrm{as}\ \varepsilon\to 0,
\end{equation}
where $\{e_1, \dots, e_n\}$ is the canonical basis of $\mathbb{R}^n$.
There are many other mean value formulas for linear elliptic  operators  other than the Laplacian, see \cite{Littman.et.al1963}, and
 for degenerate elliptic equations, see \cite{BL12}. 
 
In recent years asymptotic mean value formulas were found for nonlinear operators
such as the normalized (also called homogeneous) $p$-Laplacian, 
\[
\Delta_p^N u=
|\nabla u|^{2-p}\,\text{div}\big(|\nabla u|^{p-2}\nabla u\big)
=
 \Delta u+(p-2)\, \Delta_{\infty}^{N}u,
\]
for $1<p  < \infty$.
 These mean value formulas come from the connection 
   between probability (via the
   dynamic programming principle for  Tug-of-War games) 
   and the normalized infinity Laplacian, see \cite{LeGruyerArcher,LeGruyer,PSSW}. A 
nonlinear mean value property for $p$-harmonic functions first appeared in \cite{[Manfredi et al. 2010]}  motivated by the Random Tug-of-War games with noise in \cite{PSSW2}. 
It was proved in \cite{[Manfredi et al. 2010]}  that $p$-harmonic functions
are characterized by the fact that they satisfy the following  asymptotic mean value formula
 \begin{equation}\label{MVFormula}
u(x)
=
\left(\frac{p-2}{p+n}\right)
\left(\frac{\displaystyle \max_{{B}_\varepsilon (x)}u+\min_{B_\varepsilon  (x)}u}{2}\right)
+
\left(\frac{2+n}{p+n}\right)\dashint_{B_\varepsilon  (x)}u(y)\, dy+o(\varepsilon^2)\quad\textnormal{as}\ \ \varepsilon \to
0,
\end{equation}
in the viscosity sense. The  asymptotic mean value formula \eqref{MVFormula} holds in the viscosity sense if 
 whenever a smooth test function with non-vanishing gradient  touches $u$ from above (respectively below) at  a point $x$, the mean value formula \eqref{MVFormula} 
is satisfied with $\leq$ (respectively $\geq$) for the test function at  $x$. This is weaker than requiring the asymptotic formula to hold in the classical sense, yet enough to characterize $p$-harmonic functions. 
For mean value properties for the $p$-Laplacian in the Heisenberg group see \cite{LMan} and for the standard variational $p-$Laplacian, 
see \cite{dTLin}. See also  \cite{Angel.Arroyo.Tesis} and the recently published book \cite{[Blanc and Rossi 2019]} for historical references and  more general equations. 

It is worth noting that the expression
$$\left(\frac{p-2}{p+n}\right)
\left( \displaystyle \frac{\displaystyle \max_{{B}_\varepsilon (x)}u+\min_{B_\varepsilon  (x)}u}{\displaystyle 2}\right)
+
\left(\frac{2+n}{p+n}\right)\dashint_{B_\varepsilon  (x)}u(y)\, dy$$
in the asymptotic 
mean value property \eqref{MVFormula} has a game-theoretic interpretation for which it is essential that the coefficients are positive and  add up to 1, so that they play the role of conditional probabilities.

For asymptotic mean value properties for the elliptic Monge-Amp\`ere  equation 
\begin{equation}\label{MVP.solid.M-A.visco}
\det D^2u=f\quad \textrm{in}\ \Omega
\end{equation}
we refer to our recent paper \cite{[Blanc et al. 2020]}. For the equation 
to fit into the framework of the theory of fully nonlinear elliptic equations, one must look for 
convex solutions $u$ to ensure that $\det (D^2 u)$ is
indeed a monotone function of $D^2u$. Thus, one requires the right-hand side $f(x)$ to be non-negative. 
See \cite{Ca-Ni-Sp,Ca-Ni-Sp2}.
Moreover, 
the Monge-Amp\`ere equation can be expressed as an infimum of a family of linear operators as follows,
\begin{equation}\label{first.main.ingredient.intro.88}
 \big(\det D^2u(x)\big)^{1/n}= \frac1n \inf_{\det A=1} {\rm trace}(A^tD^2u(x)A).
\end{equation}

Let   $\phi(\varepsilon)$ 
be a positive function 
such that 
\begin{equation}\label{hipotesis.phi}
 \lim_{\varepsilon\to0} \phi(\varepsilon)=\infty
 \qquad
  \textrm{and}
 \qquad
 \lim_{\varepsilon\to0} \varepsilon\,\phi(\varepsilon) =0.
\end{equation}
Then,
a convex function $u\in C(\Omega)$ is a viscosity solution of the Monge-Amp\`ere equation 
\eqref{MVP.solid.M-A.visco}
if and only if for every $x\in \Omega$ we have
\begin{equation}\label{MVP.solid.visco}
\begin{split}
u(x)
=
\mathop{\mathop{\inf}_{\det A=1}}_{A\leq \phi(\varepsilon)I}
\bigg\{
\dashint_{B_{\varepsilon}(0)}
u(x+Ay)
\,dy
\bigg\}
-\frac{n}{2(n+2)}\,(f(x))^{1/n}\,\varepsilon^2
+o(\varepsilon^2)
\quad\textnormal{as}\ \ \varepsilon \to
0,
\end{split}
\end{equation}
in the viscosity sense, see \cite{[Blanc et al. 2020]}. 
We remark that classical $C^{2}$ solutions to the Monge-Amp\`ere equation
satisfy \eqref{MVP.solid.visco} in the standard pointwise sense.

\subsection{Asymptotic mean value formulas for parabolic equations}
In the linear case, 
$u$ is a solution to the heat equation
$$
\frac{\partial u}{\partial t} (x,t) = \Delta u(x,t)
$$ 
if and only if $u$ satisfies the mean value formula 
$$
u(x,t) = \int_{E(x,t;r)} u(y,s) \frac{|x-y|^2}{(t-s)^2} dy ds,
$$
where the integral is taken over the heat ball 
$$
E(x,t;r) = \Big\{ (y,s)\in \mathbb{R}^{n+1} : s\leq t, (4\pi (t-s))^{1/2} e^{\frac{(x-y)^2}{4n(t-s)}} \leq r \Big\}.
$$
Concerning mean value formulas for
the heat equation we refer also to \cite{Watson} and \cite{Aimar}. 
For a version with variable coefficients, see \cite{Fabes} and for a proof of the rigidity of the formula see \cite{Ko, SuWa}.

It turns out  there is a simpler asymptotic mean value formula where we only need to integrate over the parabolic cylinder
$ B_{\eps}(x)  \times (t-\frac{\eps^2}{n+2}, t)$ and there is no need for a kernel; i.e., $u$ is 
 a solution to the heat equation if and only if $u$ verifies
$$
u(x,t)=\aveint{t-\frac{\eps^2}{n+2} }{t} \kint_{B_{\eps}(x)} u(y,s) \, dy\, ds
+o(\eps^2), \quad \mbox{as }\eps \to 0.
$$
Equivalently, the asymptotic mean value formula
\[
u(x,t)= \aveint{t-\eps^2}{t}\kint_{B_\eps (x)} u(y,s)
\, dy \, ds +o(\eps^2), \quad \mbox{as } \eps \to 0,
\]
holds in the viscosity sense
if and only if $u$ is a  solution to
$$
(n+2)\, u_t (x,t)=\Delta u (x,t).
$$

For the normalized parabolic $p$-Laplacian it is convenient to write the equation in the form
\begin{equation}\label{eq:pequation}
(n+p)\,u_t (x,t) = |\nabla u|^{2-p}\Delta_p u (x,t).
\end{equation}
This equation has been studied in  \cite{MPRparab},  where a 
game approximation (that leads to an asymptotic mean value formula) was analyzed. Namely, a function $u$
solves \eqref{eq:pequation} if and only if the following expansion holds in the viscosity sense
\[
\begin{split}
u(x,t)= \frac{1}{2}
\left(\frac{p-2}{p+n}\right)
\dashint_{t-\eps^2}^{t}&\Big( \max_{y\in  B_\eps(x)}
u(y,s) + \min_{y\in B_\eps (x)} u(y,s) \Big) d s
\\
&+\left(\frac{2+n}{p+n}\right) \dashint_{t-\eps^2}^{t} \dashint_{B_\eps (x)} u(y,s)
\, d y\,  d s +o(\eps^2), 
\quad\textrm{as $\eps \to 0$.}
\end{split}
\]

\section{Main results} Our main goal is to obtain mean value formulas for parabolic versions of the Monge-Amp\`ere equation
and other nonlinear parabolic equations. 

\subsection{Parabolic Monge-Amp\`ere} 
 First we show that an asymptotic, 
 nonlinear mean value formula holds for two different parabolic versions of the Monge-Amp\`ere equation. The first one reads as follows,
 \begin{equation} \label{monge.ampere.1}
 \frac{\partial u}{\partial t} (x,t) = ( \det{(D^2u (x,t))})^{1/n} + f(x,t),
 \end{equation}
and has been studied in \cite{Chop,DasSav,Fir,Ham}
in relation to geometric evolution problems.

 Our first result describes an asymptotic mean value formula for this equation. 
 
\begin{theorem} \label{te1.intro} Let   $\phi(\varepsilon)$ 
be a positive function 
such that 
\begin{equation}\label{hipotesis.phi.intro.teo}
 \lim_{\varepsilon\to0} \phi(\varepsilon)=\infty
 \qquad
  \textrm{and}
 \qquad
 \lim_{\varepsilon\to0} \varepsilon\,\phi(\varepsilon) =0.
\end{equation}
A function $u\in C(\Omega \times (t_1,t_2))$ that is convex in the spatial variables 
is a viscosity solution of the Monge-Amp\`ere equation 
\eqref{monge.ampere.1}
if and only if 
\begin{equation}\label{MVP.pabob.1}
\begin{split}
u(x,t)
=
\mathop{\mathop{\inf}_{\det (A)=1}}_{A\leq \phi(\varepsilon)I}
\bigg\{
\dashint_{t- \frac{n}{n+2} \varepsilon^2}^t\dashint_{B_{\varepsilon}(0)}
u(x+Ay,s)
\,dy\,ds
\bigg\}
+ \frac{n}{2(n+2)}\, f(x,t)\,\varepsilon^2
+o(\varepsilon^2)
\end{split}
\end{equation}
as $\varepsilon\to0$ for $x\in \Omega$, $t \in (t_1,t_2)$,
in the viscosity sense. 
\end{theorem}

For the precise definition of a viscosity solution and the statement of a mean value formula in the viscosity sense
we refer to Section \ref{sect-prelim} (see also \cite{CIL} and \cite{[Manfredi et al. 2010]}). Informally, an equation or a mean value property
holds in the viscosity sense when it holds with the appropriate inequality instead of an equality for smooth functions that touch $u$ from above 
or from below at $(x,t)$. 

The mean value property \eqref{MVP.pabob.1} involves an average of $u$ both in space and in time. 
Notice the parabolic character of the time average in formula \eqref{MVP.pabob.1}, where we are integrating over a time 
interval of length comparable to $\varepsilon^2$.

\begin{remark}\label{remark.polar.decomposition} 
 We assume without loss of generality that the matrices $A$ that appear  
 throughout this paper are symmetric and positive definite. This assumption is not restrictive
 since we can use the (unique) left polar decomposition of $A$, namely $A = SQ$
(where $Q$ is orthogonal and $S$ is a positive definite symmetric matrix) and consider $S$ instead 
of a general matrix $A$ in our formulas.
\end{remark}

\begin{remark} \label{rem-MVP-puntual}
The concept of a mean value formula in the viscosity sense is weaker than a mean value formula that
holds in a pointwise sense. In fact, there are examples of asymptotic mean value formulas that hold in the
viscosity sense but do not hold pointwise, like the ones that hold for the infinity Laplacian, see \cite{[Manfredi et al. 2010]}, and for the elliptic 
Monge-Amp\`ere equation, see \cite{[Blanc et al. 2020]}.  
When a solution to the involved equation is smooth, $u(x,t) \in C^{2,1}$, the mean value property
holds pointwise. This is a consequence of the fact that for $C^{2,1}$ functions we have formula \eqref{MVP.pabob.1} in the classical pointwise sense. \end{remark}

\begin{remark} \label{remark-evaluando} We can also obtain a mean value property 
evaluating at $t- \frac{n}{2(n+2)} \varepsilon^2$ instead of averaging in time;  that is,
it holds that $u$ is a viscosity solution of the parabolic Monge-Amp\`ere equation 
\eqref{monge.ampere.1}
if and only if 
\begin{equation}\label{MVP.pabob.1.777}
\begin{split}
u(x,t)
=
\mathop{\mathop{\inf}_{\det (A)=1}}_{A\leq \phi(\varepsilon)I}
\bigg\{
\dashint_{B_{\varepsilon}(0)}
u\Big(x+Ay, t- \frac{n}{2(n+2)} \varepsilon^2\Big)
\,dy 
\bigg\}
+ \frac{n}{2(n+2)}\, f(x,t)\,\varepsilon^2
+o(\varepsilon^2),
\end{split}
\end{equation}
in the viscosity sense.
\end{remark}

A different version of the parabolic Monge-Amp\`ere equation reads 
 \begin{equation} \label{monge.ampere.2}
 - \frac{\partial u}{\partial t} (x,t) \cdot \mbox{det} \big(D^2u (x,t)\big) =  f(x,t)
 \end{equation}
and appears in connection with the movement
 of a hypersurface by Gauss-Kronecker curvature, see \cite{Tso}. 
 For the study of this equation we refer to \cite{GH, Xiong, Zhang} and concerning regularity of 
 the solutions to \cite{Gut, Tang, WW1, WW2}.   The
 asymptotic behavior at infinity of global solutions has been studied in \cite{WangBao}.

In order for the general viscosity theory to work in this case,  one has to 
 restrict to solutions that 
 are parabolically convex. A function $u:\Omega \times (t_1,t_2) \to \mathbb{R}$,  
 is \textit{parabolically convex}  if it is continuous, convex in $x$, and
 non-increasing in $t$. Therefore, we assume that the right-hand side $f$ is nonnegative.
 
 Our next result shows that there is also a mean value formula in this case.

\begin{theorem} \label{te2.intro} Let  $\phi(\varepsilon)$ 
be a positive function 
such that 
\begin{equation}\label{hipotesis.phi.intro.teo.33}
 \lim_{\varepsilon\to0} \phi(\varepsilon)=\infty
 \qquad
  \textrm{and}
 \qquad
 \lim_{\varepsilon\to0} \varepsilon\,\phi(\varepsilon) =0.
\end{equation}
A parabolically convex function $u\in C(\Omega \times (t_1,t_2))$ is a viscosity solution of the Monge-Amp\`ere equation 
\eqref{monge.ampere.2}
if and only if 
\begin{equation}\label{MVP.pabob.2}
\begin{split}
u(x,t)
= \!\!\!\!\!\!\!\!\!
\mathop{\mathop{\inf}_{\det (A) \times b =1}}_{A\leq \phi(\varepsilon)I,\ b \leq \phi (\varepsilon)}
\!\!\!\!\!\! \bigg\{
\dashint_{t-  \frac{b^2}{n+2} \varepsilon^2}^t\dashint_{B_{\varepsilon}(0)}
u(x+Ay, s)
\,dy\,ds
\bigg\}
-\frac{n+1}{2(n+2)}\,(f(x,t))^{\frac{1}{n+1}}\,\varepsilon^2
+o(\varepsilon^2)
\end{split}
\end{equation}
as $\varepsilon\to0$ for $x\in \Omega$, $t \in (t_1,t_2)$
in the viscosity sense. 
\end{theorem}

Notice that  formula \eqref{MVP.pabob.2} is similar to \eqref{MVP.pabob.1},  except for
the extra parameter $b$ that appears in the infimum, which is related to the  time interval average. 

\begin{remark} \label{remark-evaluando.22} We can again obtain a mean value property 
evaluating at $t- \frac{b^2}{2(n+2)} \varepsilon^2$ instead of averaging in time; that is, the 
parabolically convex function 
$u\in C(\Omega \times (t_1,t_2))$ is a viscosity solution of the Monge-Amp\`ere equation 
\eqref{monge.ampere.2}
if and only if 
\begin{equation}\label{MVP.pabob.2.22}
\begin{split}
u(x,t)
= \!\!\!\!\!\!\!\!\!
\mathop{\mathop{\inf}_{\det (A) \times b =1}}_{A\leq \phi(\varepsilon)I,\ b \leq \phi (\varepsilon)}
\!\!\! \!\!\! \bigg\{ \dashint_{B_{\varepsilon}(0)}
u\left(x+Ay, t-  \frac{b^2}{2(n+2)} \varepsilon^2\right)
\,dy\,ds \bigg\}
-\frac{n+1}{2(n+2)}\,(f(x,t))^{\frac{1}{n+1}}\,\varepsilon^2
+o(\varepsilon^2)
\end{split}
\end{equation}
as $\varepsilon\to0$ 
in the viscosity sense. 
\end{remark}

 \subsection{Bounded Operators}

\subsubsection{Infimum operators}\label{subsec.infimums}

Let us start with the differential operator $F:\Omega \times \R \times \R \times \mathbb{S}^n(\R)\to \R$ given by
\begin{equation}\label{operator.bounded.A}
F\Big(x,t,\frac{\partial u}{\partial t}, D^2u \Big)=\inf_{(A,b)\in \A_{x,t}} \Big\{ \tr(A^tD^2u(x,t)A)- b \frac{\partial u}{\partial t} (x,t) \Big\} .
\end{equation}
Here $\A_{x,t}\subset \mathbb{S}_+^n(\mathbb{R})\times \mathbb{R}^+$ is a bounded subset for each point $(x,t) \in\R^n\times \mathbb{R}^+$ and  $\mathbb{S}_+^n(\mathbb{R})$ denotes the set of symmetric positive semi-definite matrices.
Examples of these operators include parabolic equations related to Pucci operators and evolution problems associated with the convex envelope. 
See Section 5 below for details.

 
 Our next result gives an asymptotic mean value formula that characterizes viscosity solutions of the corresponding homogeneous equation.

\begin{theorem} \label{th.sol.viscosa.intro} A function $u\in C(\Omega \times (t_1,t_2))$, is a viscosity solution of the equation
\begin{equation} \label{PDE.inf}
F \Big(x,t, \frac{\partial u}{\partial t}, D^2u \Big) =0, 
\end{equation} 
if and only if
\begin{equation}
\label{eq.main}
u(x,t) = \inf_{(A,b)\in \A_{x,t}}
\dashint_{t- \frac{b}{n+2} \varepsilon^2}^t \dashint_{B_{\varepsilon}(0)}
u(x+Ay, s)
\,dy\,ds
+
o(\varepsilon^2),
\end{equation}
as $\eps\to0$ in the viscosity sense.
\end{theorem}

Observe that an analogous statement holds for suprema, and hence,  we can tackle operators of the form
\begin{equation}\label{operator.bounded.A.sup}
F\Big(x,t,\frac{\partial u}{\partial t}, D^2u \Big)=\sup_{(A,b)\in \A_{x,t}}\Big\{ \tr(A^tD^2u(x,t)A)- b \frac{\partial u}{\partial t} (x,t) \Big\} .
\end{equation}

\begin{remark}\label{remark.local} 
Since the sets $\A_{x,t}$ are bounded, the formula is local.
In fact, for every $x\in\Omega$, $t\in(t_1,t_2)$  
there exists $C=C(x,t)>0$ such that $A\leq C I$ and $b\leq C$ for every $(A,b)\in\A_{x,t}$ and we get
\[ 
\textrm{dist}(x+Ay,x)=|Ay|\leq C \varepsilon \to 0, \qquad \mbox{and} \qquad 0\leq t- s \leq  \frac{b}{n+2}\,\eps^2\le \frac{C}{n+2}\,\eps^2\to 0
\]
as $\eps\to 0$ for every $y\in B_\varepsilon(0)$ and $ s\in (t-\eps^2\frac{b\,n}{n+2},t)$.  
In particular observe that for $\eps$ small enough $(x+Ay, s)\in\Omega \times (t_1,t_2)$ for every $y\in B_\eps(0)$ and $s\in (t-\eps^2\frac{b}{n+2},t)$.
\end{remark}

\subsubsection{Supremum-Infimum operators}\label{subsec.sup-infimums}

Our next step is to consider a special type of Isaacs operators where the supremum (or the infimum) is taken over 
a subset $\mathbb{A}_{x,t}$ of the parts of $\mathbb{S}_+^n(\mathbb{R})\times\mathbb{R}^+$. To be more precise, 
let $\mathbb{A}_{x,t}\subset \mathcal P (\mathbb{S}_+^n(\mathbb{R})\times\mathbb{R}^+)$ be a subset for each $x\in\Omega$, $t\in (t_1,t_2)$ such that 
\[
\bigcup\mathbb{A}_{x,t}
=
\Big\{
(A,b)\in \mathbb{S}_+^n(\mathbb{R})\times\mathbb{R}^+\ : \
(A,b)\in \mathcal{A}\ \textrm{for some}\ \mathcal{A}\in\mathbb{A}_{x,t}
\Big\}
\]
 is bounded.
We consider the differential operator $F:\Omega \times \R \times \R \times \mathbb{S}^n(\R)\to \R$ given by
\begin{equation}\label{isasc.pablo.intro}
F\Big(x, t, \frac{\partial u}{\partial t}, D^2u \Big) =\sup_{\mathcal{A}\in\mathbb{A}_{x,t}} \inf_{(A,b)\in \A }
\Big\{ \textrm{trace}(A^tD^2u(x,t)A)
- b \frac{\partial u}{\partial t} (x,t) \Big\}.
\end{equation}
Examples of these operators include evolution problems for the eigenvalues of the Hessian. See Section 5 below for  details.

\begin{theorem} \label{th.sol.viscosa.intro.22} A function $u\in C(\Omega \times (t_1,t_2))$, is a viscosity solution to
\begin{equation} \label{PDE.inf.22}
F\Big(x, t, \frac{\partial u}{\partial t}, D^2u \Big)  =0, 
\end{equation} 
if and only if
\begin{equation}\label{sup.inf.MVP.formula}
u(x,t) = \sup_{\mathcal{A}\in\mathbb{A}_{x,t}} \inf_{(A,t)\in \A }
\dashint_{t- \frac{b}{n+2}\varepsilon^2}^t \dashint_{B_{\varepsilon}(0)}
u(x+Ay, s)
\,dy\, ds
+
o(\varepsilon^2),
\end{equation}
as $\eps\to0$ in the viscosity sense.
\end{theorem}

\begin{remark} \label{remark-evaluando.77} In the previous two cases we can 
evaluate at $t-  \frac{b}{2(n+2)} \varepsilon^2$ instead of taking averages in time and get a mean value formula of the form
\begin{equation}
\label{pepe}
u(x,t) =\sup_{\mathcal{A}\in\mathbb{A}_{x,t}} \inf_{(A,t)\in \A }
 \dashint_{B_{\varepsilon}(0)}
u \Big(x+Ay, t- \frac{b}{2(n+2)} \varepsilon^2 \Big)
\,dy\, ds
+
o(\varepsilon^2),\quad \textrm{as $\varepsilon\to0$.}
\end{equation}
\end{remark}

The paper is organized as follows: In Section \ref{sect-prelim} we gather some definitions and preliminary results;
in Section \ref{sect-parab-MA} we prove the mean value formulas for the two versions of the parabolic Monge-Amp\`ere
equation;
in Section \ref{sec.infimums} we deal with infimum and sup-inf operators;
finally, Section \ref{sec-DPP}, contains a brief discussion of the relation between these mean value formulas and 
Dynamic Programming Principles from game theory.

\section{Preliminaries} \label{sect-prelim}

In this section we set the notation, recall basic results on the Monge-Amp\`ere equation,  and state some definitions.
We begin by stating the definition of a viscosity solution to a fully nonlinear second order parabolic PDE.
We refer to
\cite{CIL} for general results on viscosity solutions.

Given a continuous function
\[
F:\Omega\times\R \times \R\times\mathbb{S}^n(\R)\to\R
\]
where $\mathbb{S}^n(\R)$ denotes the set of symmetric $n\times n$ matrices,
 we consider the PDE 
\begin{equation}
\label{eqvissol}
F\left(x,t,\frac{\partial u}{\partial t} (x,t), D^2u (x,t) \right) =0, \qquad x \in \Omega,\ t \in (t_1,t_2).
\end{equation}
Viscosity solutions use the monotonicity of $F$ in $D^2u$ (ellipticity) and in $\frac{\partial u}{\partial t} $ in order to 
evaluate the equation for smooth test functions and obtain sub- and supersolution  inequalities. 

\begin{definition}
\label{def.sol.viscosa.1}
A lower semi-continuous function $ u $ is a viscosity
supersolution of \eqref{eqvissol} if for every $ \phi \in C^{2,1}$ such that $ \phi $
touches $u$ at $(x,t) \in \Omega \times (t_1,t_2)$ strictly from below (that is, $u-\phi$ has a strict minimum at $(x,t)$ 
with $u(x,t) = \phi (x,t)$), we have
$$F \Big(x,t, \frac{\partial \phi}{\partial t} (x,t),D^2\phi(x,t)\Big)\geq 0.$$
An upper semi-continuous function $u$ is a subsolution of \eqref{eqvissol} if
for every $ \phi \in C^{2,1}$ such that $\phi$ touches $u$ at $ (x,t) \in
\Omega \times (t_1,t_2)$ strictly from above (that is, $u-\phi$ has a strict maximum at $(x,t)$ with $u(x,t) = \phi(x,t)$), we have
$$F \Big(x,t, \frac{\partial \phi}{\partial t} (x,t),D^2\phi(x,t)\Big)\leq 0.$$
Finally, $u$ is a viscosity solution of \eqref{eqvissol} if it is both a super- and a subsolution.
\end{definition}

To apply this definition to the parabolic Monge-Amp\`ere equation \eqref{monge.ampere.1} we have to consider the operator
\[
F\Big(x,t,\frac{\partial u}{\partial t} (x,t), D^2u (x,t) \Big)
=
\begin{cases}
 \displaystyle \frac{\partial u}{\partial t} (x,t) - ( \det{(D^2u (x,t))})^{1/n} - f(x,t) 
&\text { if } D^2u (x,t)\geq 0\\[10pt]
-\infty
&\text { otherwise.}
\end{cases}
\]
This is equivalent to requiring the function to be convex in the space variable and restricting the test functions to paraboloids convex in space.
Similarly, for solutions to equation \eqref{monge.ampere.2} we require the function to be parabolically convex (convex in $x$ and  non-increasing in $t$) and restrict the test functions to parabolically convex paraboloids, defined next.

\begin{definition}\label{def.paraboloides}
A $C^{2,1}$ function $P(x,t)$ is a parabolic paraboloid if and only if it coincides with its second order Taylor expansion in $x$ and first order in $t$, 
i.e., we have
\[
P(x,t)=P(x_0,t_0)+ \frac{\partial P}{\partial t} (x_0,t_0) (t-t_0) + \langle\nabla P(x_0,t_0),(x-x_0)\rangle+\frac12\langle D^2 P(x_0,t_0)(x-x_0),(x-x_0)\rangle
\]
for any given $(x_0,t_0)$. Furthermore, $P(x,t)$ is a parabolically convex paraboloid if and only if 
$\frac{\partial P}{\partial t} \geq 0$ and $D^2P\geq 0$.
\end{definition}

We will also need the definition of an asymptotic mean value formula in the viscosity sense.
First, recall that given a constant $c$ and a real function $g$ we write
$$c\le g(\varepsilon) + o(\varepsilon^2) \quad \text{ as } \varepsilon\to 0$$ whenever we have
$$\lim_{\varepsilon\to 0} \frac{\left[ c-g(\varepsilon)\right]^+}{\varepsilon^2}=0, $$
and 
$$c\ge g(\varepsilon) + o(\varepsilon^2) \quad \text{ as } \varepsilon\to 0$$ whenever we have
$$\lim_{\varepsilon\to 0} \frac{\left[ c-g(\varepsilon)\right]^-}{\varepsilon^2}=0.$$

In the next definition $M (u, \varepsilon ) (x)$ stands for a mean value operator (that depends on the parameter
$\varepsilon$) applied to $u$ at the point $x$. As an example, consider
$$
M (u, \varepsilon ) (x,t) = \inf_{(A,b)\in\A_{x,t}} \dashint_{t-b\,\eps^2}^t
\dashint_{B_{\varepsilon}(0)}
u(x+Ay,s)
\,dy\,ds.
$$

\begin{definition} \label{def.sol.viscosa.asymp}
A continuous function  $u$  verifies
$$
u(x,t) = M (u, \varepsilon ) (x,t) + o(\varepsilon^2), \quad \mbox{as }\varepsilon \to 0,
$$
\emph{in the viscosity sense} if
\begin{enumerate}
\item for every $ \phi\in C^{2,1}$ such that $ u-\phi $ has a strict
minimum at the point $(x,t) \in \Omega \times (t_1,t_2)$  with $u(x,t)=
\phi(x,t)$, we have
$$
0 \geq - \phi (x,t) + M (\phi, \varepsilon ) (x,t) + o(\varepsilon^2).
$$

\item for every $ \phi \in C^{2,1}$ such that $ u-\phi $ has a
strict maximum at the point $ (x,t) \in {\Omega} \times (t_1,t_2)$ with $u(x,t)=
\phi(x,t)$, we have
$$
0 \leq - \phi (x,t) + M (\phi, \varepsilon ) (x,t)
 + o(\varepsilon^2).
$$
\end{enumerate}
\end{definition}

The following elementary fact will be used several times in the sequel. 
\begin{lemma}\label{lemma.trace.integral}
Let $M$ be a square matrix of dimension $n$. Then,
\begin{align}
{\rm trace}(M)
&=
\frac{n}{\varepsilon^2}
\dashint_{\partial B_\varepsilon (0)}
\langle My,y\rangle
\,d\mathcal{H}^{n-1}(y),
\label{trace.surface.integral.representation}
\\
&=
\frac{n+2}{\varepsilon^2}
\dashint_{B_\varepsilon (0)}
\langle My,y\rangle
\,dy,
\label{trace.solid.integral.representation}
\end{align}
\end{lemma}

For symmetric square matrices, $A>0$ means positive definite and $A\geq0$ means positive semidefinite. We will denote $\lambda_i(A)$ the eigenvalues of $A$, in particular $\lambda_{\min}(A)$ and $\lambda_{\max}(A)$ are the smallest and largest eigenvalues, respectively. We  have the following linear algebra facts.

\begin{lemma}\label{caract.determ}
Let $M$ be symmetric and $M\geq 0$. Then,
\[
\inf_{\det A=1}{\rm trace}(A^tMA)=n (\det(M))^{1/n}.
\]
On the other hand, if $M$ has negative eigenvalues, then the  infimum is $-\infty$.
\end{lemma}

For a proof of Lemma \ref{caract.determ} we refer to \cite{[Blanc et al. 2020]}.

\begin{lemma}\label{lemma.uniform.ellipticity.visco}
Let $M >0$. Then, for every 
\begin{equation}  \label{lemma.uniform.ellipticity.visco.theta0}
\theta> \theta_0:=\left(\frac{\left(\det{M}\right)^{1/n}}{\lambda_{\min} (M)}\right)^{1/2},
\end{equation}
we have
\begin{equation}  \label{lemma.uniform.ellipticity.visco.statement}
 \inf_{\det A=1}{\rm trace}\big(A^t M A\big)
 =
 \mathop{\mathop{\inf}_{\det A=1}}_{A\leq \theta I}
{\rm trace}(A^t M A) .
\end{equation}
\end{lemma}

\begin{proof}
In general, the right-hand side is larger than the left-hand side as we are taking infimum over an smaller set.
The infimum is realized for $A^*=\left(\det{M}\right)^{\frac{1}{2n}} M^{-1/2}$.
Then the result follows since $A^*\leq \theta_0 I$.
\end{proof}

\begin{lemma}
For $M> 0$ we have 
\begin{equation}
\label{determinant.estimates1}
\left(\det\left(M\pm\eta I\right)\right)^{1/n}
=
\left(\det{M}\right)^{1/n}
+O(\eta)
\qquad
\textrm{as $\eta\to0$}.
\end{equation}
For $M\geq 0$ we have 
\begin{equation}
\label{determinant.estimates2}
\left(\det\left(M+\eta I\right)\right)^{1/n}
=
\left(\det{M}\right)^{1/n}
+O(\eta)
\qquad
\textrm{as $\eta\to0$}.
\end{equation}

\end{lemma}

\begin{proof}
To see this, first notice that
\begin{equation}\label{det.expansion.formula}
\det\left(M\pm\eta I\right)
=
\det{M}
+
\sum_{k=1}^{n}(\pm \eta)^{k}\sigma_{n-k}(M),
\end{equation}
where the coefficients $\sigma_{n-k}(M)$ in the expansion are given by the 
elementary symmetric polynomials on the eigenvalues of $M$, which are positive.
Therefore, 
\[
\left(
\det{M}-C\eta
\right)^{1/n}
\leq
\left(
\det\left(M\pm\eta I\right)
\right)^{1/n}
\leq
\left(
\det{M} + C\eta
\right)^{1/n}
\]
for some $C>0$.
Then, the Mean Value Theorem applied to $g(t)=t^{1/n}$, with $a=\det M$ and $b=\det M+C\eta$ gives that there 
exists $\xi\in(a,b)$ such that
\[
\left(
\det M + C\eta
\right)^{1/n}=
g(b)
=g(a)+g'(\xi)(b-a)
\leq
\left(
\det{M}
\right)^{1/n}
+C\eta.
\]
Similarly, we obtain that
\[
\left(
\det{M} - C\eta
\right)^{1/n}
\geq
\left(
\det{M}
\right)^{1/n}
-C\eta
\]
and \eqref{determinant.estimates1} follows.

It remains to prove \eqref{determinant.estimates2} in the case $\det M=0$.
In this case the upper bound obtained is still at our disposal.
Therefore we conclude by observing that $\left(\det\left(M+\eta I\right)\right)^{1/n}\geq 0$.
\end{proof}

\begin{lemma}\label{lemma.convergence.eta}
For $M> 0$ we have 
\begin{equation}
\label{estimates1}
\mathop{\mathop{\inf}_{\det A=1}}_{A\leq \phi(\varepsilon)I}
{\rm trace}(A^t (M\pm\eta I) A)
\to
n\left(\det{M}\right)^{1/n} \qquad
\textrm{as $\eps,\eta\to0$}.
\end{equation}
For $M\geq 0$ we have 
\begin{equation}
\label{estimates2}
\mathop{\mathop{\inf}_{\det A=1}}_{A\leq \phi(\varepsilon)I}
{\rm trace}(A^t (M+\eta I) A)
\to
n\left(\det{M}\right)^{1/n}  \qquad
\textrm{as $\eps,\eta\to0$}.
\end{equation}
\end{lemma}

\begin{proof}
For $M>0$, we consider $\eta<\min\{1,\lambda_{\min}(M)/2\}$, we have
\[
\left(\frac{\left(\det{M \pm\eta I}\right)^{1/n}}{\lambda_{\min} (M \pm\eta I)}\right)^{1/2}
\leq
\left(\frac{\left(\det{M +I}\right)^{1/n}}{\lambda_{\min} (M)/2)}\right)^{1/2}
=\theta_0.
\]
Then, for $\eps_0$ such that for every $\eps<\eps_0$ we have $\phi(\eps)>\theta_0$ combining Lemma~\ref{lemma.uniform.ellipticity.visco} and Lemma~\ref{caract.determ} we get
\[
\mathop{\mathop{\inf}_{\det A=1}}_{A\leq \phi(\varepsilon)I}
{\rm trace}(A^t (M\pm\eta I) A)
=
n \left(\det{M\pm\eta I}\right)^{1/n}.
\]
The result follows from equation \eqref{determinant.estimates1}.

It remains to prove the case $\det(M)=0$.
Given $\delta>0$, by Lemma~\ref{caract.determ} there exists $A_0$ with $\det(A_0)=1$ such that
\[
{\rm trace}(A_0^t M A_0)<\delta.
\]
Then, for $\eps$ such that $A_0\leq \phi(\eps)I$ we have
\[
\mathop{\mathop{\inf}_{\det A=1}}_{A\leq \phi(\varepsilon)I}
{\rm trace}(A^t (M\pm\eta I) A)
\leq {\rm trace}(A_0^t (M+\eta I) A_0)
\]
and by \eqref{determinant.estimates2} for $\eta$ small enough we get
\[
{\rm trace}(A_0^t (M+\eta I) A_0)<\delta
\]
and the result follows.
\end{proof}

\section{Parabolic Monge-Amp\`ere} \label{sect-parab-MA}

\subsection{First version} Let us start with the equation
 \begin{equation} \label{monge.ampere.1.s}
 \frac{\partial u}{\partial t} (x,t) = ( \det{(D^2u (x,t))})^{1/n} + f(x,t).
 \end{equation}
 Our goal is to show that $u$ is a
viscosity solution of the Monge-Amp\`ere equation 
\eqref{monge.ampere.1.s}
if and only if 
\begin{equation}
\begin{split}
u(x,t)
=
\mathop{\mathop{\inf}_{\det (A)=1}}_{A\leq \phi(\varepsilon)I}
\bigg\{
\dashint_{t- \frac{n}{n+2} \varepsilon^2}^t\dashint_{B_{\varepsilon}(0)}
u(x+Ay,s)
\,dy\,ds
\bigg\}
+\frac{n}{2(n+2)}\, f(x,t)\,\varepsilon^2
+o(\varepsilon^2)
\end{split}
\end{equation}
as $\varepsilon\to0$ for $x\in \Omega$
in the viscosity sense.

We first prove the asymptotic mean value property under the restriction that solutions of the Monge-Amp\`ere equation are classical ($C^{2,1}$ solutions) and then we deal with the general case (viscosity solutions). 
We start by proving the result for smooth strictly convex functions (with strictly positive definite Hessian).
We proceed with the proof of  Theorem \ref{te1.intro}.

\begin{proof}[Proof of Theorem \ref{te1.intro}]
As we have mentioned, first we prove the result for classical solutions.
Assume that $u$ is $C^{2,1}$, convex in space and   a solution to 
\begin{equation} \label{monge.ampere.1.s.77}
 \frac{\partial u}{\partial t} (x,t) = ( \det{(D^2u (x,t))})^{1/n} + f(x,t).
 \end{equation}

We use the Taylor expansion of $u(y,s)$, given by
$$
\begin{array}{l}
\displaystyle
u(y,s) = u(x,t) + \frac{\partial u}{\partial t} (x,t) (s-t)  +
\langle \nabla u(x,t), (y-x) \rangle \\[10pt]
\qquad \qquad 
\displaystyle + \frac12 \langle D^2 u (x,t) (y-x), (y-x) \rangle + o(|t-s| + |x-y|^2),
\end{array}
$$
to define the parabolic paraboloid
\[
P(y,s)=u(x,t) + \frac{\partial u}{\partial t} (x,t) (s-t)  +
\langle \nabla u(x,t), (y-x) \rangle + \frac12 \langle D^2 u (x,t) (y-x), (y-x) \rangle.
\]
 Since $u\in C^{2,1}$, we have
\[
u(y,s)-P(y,s)=o(|y-x|^2 + |t-s|) \qquad\textrm{as}\ y \to x, \, s\to t,
\]
which means that for every $\eta >0$, there exists $\delta>0$ such that 
\begin{equation} \label{u-P.o.pequena}
P(y,s)-\eta\left(\frac{|y-x|^2}{2}+|t-s|\right)
\leq
u(y,s)
\leq
P(y,s)+\eta\left(\frac{|y-x|^2}{2}+|t-s|\right)
\end{equation}
for every $(y,s)\in B_\delta(x)\times(t-\delta,t+\delta)$.
For convenience, we denote
\[
P_\eta^\pm(y,s)=P(y,s)\pm\eta\left(\frac{|y-x|^2}{2}+|t-s|\right).
\]

Let us assume first that $D^2u(x,t)>0$.
We have
\[
\begin{split}
 \dashint_{t- \frac{n}{n+2}\varepsilon^2}^t&\dashint_{B_{\varepsilon}(0)}
P_\eta^\pm (x+Ay,s)
\,dy\,ds  - u (x,t)
\\
&= \dashint_{t- \frac{n}{n+2}\varepsilon^2}^t\dashint_{B_{\varepsilon}(0)} \left(\frac{\partial u}{\partial t} (x,t) (s-t) \pm \eta|t-s| \right)  \,dy\,ds
\\
&\qquad \qquad + \dashint_{t- \frac{n}{n+2} \varepsilon^2}^t\dashint_{B_{\varepsilon}(0)} 
\left(
\langle \nabla u(x,t), Ay \rangle + \frac12 \langle D^2 u (x,t) Ay, Ay \rangle \pm\frac\eta2 |Ay|^2
\right)
\,dy\,ds
\\
&= \left(-\frac{\partial u}{\partial t} (x,t) \pm \eta\right)
\dashint_{t-\frac{n}{n+2}\varepsilon^2}^t (s-t)  \,ds 
  +\frac12 \dashint_{B_{\varepsilon}(0)}  \big\langle A^t (D^2 u (x,t) \pm\eta I)Ay, y \big\rangle
\,dy  
\\
& =  \frac{n}{2(n+2)} \,\varepsilon^2 \left(-\frac{\partial u}{\partial t} (x,t) \pm\eta\right)
 +\frac12\dashint_{B_{\varepsilon}(0)}  \big\langle A^t (D^2 u (x,t) \pm\eta I)  Ay, y \big\rangle
\,dy.
\end{split}
\]
By Lemma~\ref{lemma.trace.integral} we get
$$
\dashint_{B_{\varepsilon}(0)} \frac12 \langle A^t \left(D^2u(x,t)\pm\eta I\right) Ay, y \rangle
\,dy   = 
\frac{\varepsilon^2}{2(n+2)}\,\textrm{trace}\left(A^t \left(D^2u(x,t)\pm\eta I\right) A\right),
$$
and hence, we obtain,
\[
\begin{split}
 \dashint_{t- \frac{n}{n+2}\varepsilon^2}^t\dashint_{B_{\varepsilon}(0)}&
P_\eta^\pm (x+Ay,s)
\,dy\,ds - u (x,t)
\\
&=\frac{n}{2(n+2)}\, \varepsilon^2 \left(-\frac{\partial u}{\partial t} (x,t) \pm\eta\right)
 + \frac{\varepsilon^2}{2(n+2)}\,\textrm{trace} ( A^t \left(D^2u(x,t)\pm\eta I\right) A ).
\end{split}
\]
Therefore, 
\[
\begin{split}
\mathop{\mathop{\inf}_{\det (A)=1}}_{A\leq \phi(\varepsilon)I} &
\dashint_{t- \frac{n}{n+2}\varepsilon^2}^t
\dashint_{B_{\varepsilon}(0)}
P_\eta^\pm (x+Ay,s)
\,dy\,ds
 - u(x,t)  
\\
&= 
\frac{n}{2(n+2)}\, \varepsilon^2 \left(-\frac{\partial u}{\partial t} (x,t) \pm\eta\right)
 +\frac{\varepsilon^2}{2(n+2)}\, \mathop{\mathop{\inf}_{\det (A)=1}}_{A\leq \phi(\varepsilon)I} 
 \textrm{trace}\left(A^t\left(D^2u(x,t) \pm \eta I \right)A\right).
\end{split}
\]

Observe that  $A\leq \phi( \varepsilon)I$ and $|y|\leq \varepsilon$ imply
 $x+Ay\in B_\delta(x)$  for $\varepsilon< \varepsilon_0$ (since $\eps\phi(\eps)\to 0$ as $\eps\to 0$). 
Therefore, given $\eta>0$ if $\varepsilon<\varepsilon_0$, we have
 \begin{equation}\label{C2.case.second.part}
P_\eta^-(x+Ay,s)
\leq
u(x+Ay,s)
\leq
P_\eta^+(x+Ay,s)\qquad\textrm{for every}\ y\in B_\varepsilon (0).
\end{equation}
Then, 
we get
\begin{equation}\label{C2.case.final.loop.first.part}
\begin{split}
\frac{n\,\varepsilon^2}{2(n+2)}\, &\Bigg\{\left(-\frac{\partial u}{\partial t} (x,t) -\eta\right)
 +\frac{1}{n}\, \mathop{\mathop{\inf}_{\det (A)=1}}_{A\leq \phi(\varepsilon)I} 
 \textrm{trace}\left(A^t\left(D^2u(x,t) - \eta I \right)A\right)\Bigg\}
\\
&\leq
\mathop{\mathop{\inf}_{\det A=1}}_{
A\leq \phi(\varepsilon)I}
\dashint_{t- \frac{n}{n+2}\varepsilon^2}^t \dashint_{B_{\varepsilon}(0)}
\left(
u(x+Ay,s)-u(x,t)
\right)
\,dy\,ds
\\
&\qquad\qquad\leq
\frac{n\,\varepsilon^2}{2(n+2)}\, \Bigg\{\left(-\frac{\partial u}{\partial t} (x,t) +\eta\right)
 +\frac{1}{n}\, \mathop{\mathop{\inf}_{\det (A)=1}}_{A\leq \phi(\varepsilon)I} 
 \textrm{trace}\left(A^t\left(D^2u(x,t) + \eta I \right)A\right)\Bigg\}.
\end{split}
\end{equation}
From here, the result follows from Lemma \ref{lemma.convergence.eta} since both the upper and lower bound involve an expression that converges to 
\[
-\frac{\partial u}{\partial t} (x,t) + ( \det{(D^2u (x,t))})^{1/n} = -  f(x,t)
\]
as $\eta\to 0$.

Now, if $\det(D^2u(x,t))=0$, then we use a minor modification of the above argument.
The upper bound is still at our disposal but we have to obtain a new lower bound.
Here we use the convexity in space of $u$ to get
\[
\begin{split}
\mathop{\mathop{\inf}_{\det (A)=1}}_{A\leq \phi(\varepsilon)I}
\bigg\{
\dashint_{t- \frac{n}{n+2}\varepsilon^2}^t&\dashint_{B_{\varepsilon}(0)}
u(x+Ay,s)
\,dy\,ds
\bigg\} - u(x,t)
\\
&\geq
\dashint_{t- \frac{n}{n+2}\varepsilon^2}^t
u(x,s)
\,ds
 - u(x,t)
\geq \displaystyle  
-\frac{n}{2(n+2)}\,\varepsilon^2 \left( \frac{\partial u}{\partial t} (x,t) 
  \right) + o(\varepsilon^2)
\end{split}
\]
and the result follows.

For the viscosity case, it is enough 
to use convex (in space) smooth functions as test functions in the definition of viscosity solution.  Let $u(x,t)$ be a viscosity solution to 
$$ 
\frac{\partial u}{\partial t} (x,t) = ( \det{(D^2u (x,t))})^{1/n} + f(x,t).
$$ 
Take $\phi(x,t)$ a $C^{2,1}$ function, convex in space that touches $u$ from above at
$(x,t)$, that is, we have
$$
u(y,s)- \phi(y,s) \leq u(x,t)- \phi(x,t), \qquad \mbox{for every } |y-x|\leq a ,\ t-a\leq s \leq t.
$$
The fact that $u$ is a viscosity solution implies
$$ 
\frac{\partial \phi}{\partial t} (x,t) - ( \mbox{det} (D^2 \phi (x,t)))^{1/n} \leq  f(x,t).
$$
On the other hand, from our previous computations, using that $\phi \in C^{2,1}$, we obtain
\[
\begin{split}
&\mathop{\mathop{\inf}_{\det (A)=1}}_{A\leq \phi(\varepsilon)I} \dashint_{t- \frac{n}{n+2}\varepsilon^2}^t\dashint_{B_{\varepsilon}(0)}
\phi(x+Ay,s)
\,dy\,ds  - \phi(x,t) \\
& = -\frac{n}{2(n+2)}\,\varepsilon^2 \left( \frac{\partial \phi}{\partial t} (x,t) 
 - (\det(D^2 \phi (x,t)))^{1/n} \right) + o(\varepsilon^2)
 \geq
-\frac{n}{2(n+2)}\,\varepsilon^2 f(x,t) + o(\varepsilon^2). 
\end{split}
\]

An analogous computation shows that $\phi(x,t)$ a $C^{2,1}$ convex function that touches $u$ from below at
$(x,t)$ verifies 
$$
\mathop{\mathop{\inf}_{\det (A)=1}}_{A\leq \phi (\varepsilon)I} \displaystyle
\dashint_{t- \frac{n}{n+2}\varepsilon^2}^t\dashint_{B_{\varepsilon}(0)}
\phi (x+Ay,s)
\,dy\,ds  - \phi (x,t)  \leq
-\frac{n}{2(n+2)}\,\varepsilon^2 f(x,t) + o(\varepsilon^2). 
$$
This proves that a viscosity solution to the PDE verifies the asymptotic mean value formula in the viscosity
sense.

For the converse, let $u$ be a convex function that verifies the asymptotic mean value property
in the viscosity sense. Take $\phi (x,t)$ a convex smooth function that touches $u$ from above at
$(x,t)$, then we have 
$$
u(y,s)- u(x,t) \leq \phi (y,s) - \phi (x,t), \qquad \mbox{for every } |y-x|\leq a ,\ t-a\leq s \leq t.
$$
Since $u$ verifies the asymptotic mean value property
in the viscosity sense, we have
$$
\dashint_{t- \frac{n}{n+2}\varepsilon^2}^t\dashint_{B_{\varepsilon}(0)}
\phi(x+Ay,s)
\,dy\,ds  - \phi(x,t)  \geq
-\frac{n}{2(n+2)}\,\varepsilon^2 f(x,t) + o(\varepsilon^2).
$$
Using again our previous computations for a $C^{2,1}$ function, we obtain
$$ 
-\frac{n}{2(n+2)}\,\varepsilon^2 \left( \frac{\partial \phi}{\partial t} (x,t) 
 - (\det(D^2 \phi (x,t)))^{1/n} \right) \geq
-\frac{n}{2(n+2)}\,\varepsilon^2 f(x,t) + o(\varepsilon^2)
$$
Taking the limit as $\varepsilon \to 0$ we obtain
$$
 \left( \frac{\partial \phi }{\partial t} (x,t) 
 - (\det(D^2 \phi (x,t)))^{1/n} + f(x,t) \right)  \geq 0.
$$
This proves that $u$ is a viscosity subsolution to the parabolic Monge-Amp\`ere equation. 

To prove that $u$ is a viscosity supersolution we proceed in a similar way, touching $u$ from below with a  convex smooth function $\phi$ and using the asymptotic mean value property with the reverse inequality.
\end{proof}

\begin{remark}
Using \eqref{trace.surface.integral.representation}, 
$$
{\rm trace}(M)
=
\frac{n}{\varepsilon^2}
\dashint_{\partial B_\varepsilon (0)}
\langle My,y\rangle
\,d\mathcal{H}^{n-1}(y),
$$
instead of \eqref{trace.solid.integral.representation},
we can obtain an asymptotic  mean value formula involving means over surfaces of ellipsoids:
$u$ is a
viscosity solution of the parabolic Monge-Amp\`ere equation 
\eqref{monge.ampere.1.s}
if and only if 
\begin{equation}\label{MVP.pabob.1.999.456}
\begin{split}
u(x,t)
=
\mathop{\mathop{\inf}_{\det (A)=1}}_{A\leq \phi(\varepsilon)I}
\bigg\{
\dashint_{t- \varepsilon^2}^t\dashint_{\partial B_{\varepsilon}(0)}
u(x+Ay,s)
\,dy\,ds
\bigg\}
+\frac{1}{2}\, f(x,t)\,\varepsilon^2
+o(\varepsilon^2)
\end{split}
\end{equation}
as $\varepsilon\to0$ for $x\in \Omega$ and $t \in (t_1,t_2)$,
in the viscosity sense. 
\end{remark}

\begin{remark} We can also deal with equations with coefficients such as 
$$
a(x,t) \frac{\partial u}{\partial t} (x,t) =  b(x,t)(\det (D^2 u (x,t))^{1/n} + c(x,t) f(x,t)
$$
(see \cite{DasSav}). In this case we obtain an asymptotic  mean value property of the form 
\[
\begin{split}
u(x,t)
=
\mathop{\mathop{\inf}_{\det (A)=1}}_{A\leq \phi(\varepsilon)I}
\bigg\{
\dashint_{t-  k_1(x,t) \varepsilon^2}^t\dashint_{B_{\varepsilon}(0)}
u(x+Ay,s)
\,dy\,ds
\bigg\}
+\frac{n}{2(n+2)}\, k_2(x,t) f(x,t)\,\varepsilon^2
+o(\varepsilon^2)
\end{split}
\]
as $\varepsilon\to 0$ for $x\in \Omega$ and $t \in (t_1,t_2)$ with an appropriate choice of $k_1(x,t)$ and $k_2(x,t)$. 
\end{remark}

\subsection{Second version} Now we deal with
 \begin{equation} \label{monge.ampere.2.999}
 - \frac{\partial u}{\partial t} (x,t) \cdot \mbox{det} \big(D^2u (x,t)\big) =  f(x,t)
 \end{equation}
and our goal is to prove Theorem \ref{te2.intro}, that is, to show that viscosity solutions to \eqref{monge.ampere.2.999} are characterized by 
the asymptotic  mean value formula 
\begin{equation}\label{MVP.pabob.2.678}
\begin{split}
u(x,t)
=
\mathop{\mathop{\inf}_{\det (A) \times b =1}}_{A\leq \phi(\varepsilon)I,\ b \leq \phi (\varepsilon)}
\bigg\{
\dashint_{t-  \frac{b^2}{n+2} \varepsilon^2}^t\dashint_{B_{\varepsilon}(0)}
u(x+Ay, s)
\,dy\,ds
\bigg\}
-\frac{n+1}{2(n+2)}\,f(x,t)^\frac{1}{n+1}\,\varepsilon^2
+o(\varepsilon^2)
\end{split}
\end{equation}
as $\varepsilon\to0$ for $x\in \Omega$
in the viscosity sense.

\begin{proof}[Proof of Theorem \ref{te2.intro}]
Let us first prove the result for classical
solutions. Assume that $u$ is $C^{2,1}$, convex in space, non-increasing in $t$, and  
a solution to 
\begin{equation} \label{monge.ampere.1.s.77.99}
- \frac{\partial u}{\partial t} (x,t)\cdot  \det{\big(D^2u (x,t)\big)} = f(x,t).
 \end{equation}

As in the  proof of Theorem \ref{te1.intro}, we define the parabolic paraboloid
\[
P(y,s)=u(x,t) + \frac{\partial u}{\partial t} (x,t) (s-t)  +
\langle \nabla u(x,t), (y-x) \rangle + \frac12 \langle D^2 u (x,t) (y-x), (y-x) \rangle.
\]
 Since $u\in C^{2,1}$,  the Taylor expansion of $u(y,s)$ yields
\[
u(y,s)-P(y,s)=o(|y-x|^2 + |t-s|) \qquad\textrm{as}\ y \to x, \, s\to t,
\]
which means that for every $\eta >0$, there exists $\delta>0$ such that 
\begin{equation} \label{u-P.o.pequena.2}
P(y,s)-\eta\left(\frac{|y-x|^2}{2}+|t-s|\right)
\leq
u(y,s)
\leq
P(y,s)+\eta\left(\frac{|y-x|^2}{2}+|t-s|\right)
\end{equation}
for every $(y,s)\in B_\delta(x)\times(t-\delta,t+\delta)$.
For convenience, we denote
\[
P_\eta^\pm(y,s)=P(y,s)\pm\eta\left(\frac{|y-x|^2}{2}+|t-s|\right).
\]

We assume first that $D^2u(x,t) >0$ and that $- \displaystyle \frac{\partial u}{\partial t} (x,t)>0$. In this case we have
\[
\begin{split}
\dashint_{t-  \frac{b^2}{n+2} \varepsilon^2}^t&\dashint_{B_{\varepsilon}(0)}
P_\eta^\pm (x+Ay, s)
\,dy\,ds
  - u (x,t) \\
 =& \dashint_{t- \frac{ b^2}{n+2} \varepsilon^2}^t\dashint_{B_{\varepsilon}(0)} \left(\frac{\partial u}{\partial t} (x,t) (s-t) \pm \eta|t-s| \right) \,dy\,ds
\\
&+ \dashint_{t-  \frac{b^2}{n+2}\varepsilon^2}^t\dashint_{B_{\varepsilon}(0)} \left(\langle \nabla u(x,t), Ay \rangle + \frac12 \langle D^2 u (x,t) Ay, Ay \rangle \pm \frac\eta2 |Ay|^2\right)
\,dy\,ds 
\\
=&  \left(-\frac{\partial u}{\partial t} (x,t) \pm \eta\right) \dashint_{t-  \frac{b^2}{n+2}\,\varepsilon^2}^t (t-s)  \,ds
   +\frac12 \dashint_{B_{\varepsilon}(0)}  \big\langle A^t (D^2 u (x,t) \pm\eta I)Ay, y \big\rangle
\,dy  
\\
=& \; \frac{b^2}{2(n+2)}\,\varepsilon^2 \left(-\frac{\partial u}{\partial t} (x,t) \pm \eta\right)  
+\frac12 \dashint_{B_{\varepsilon}(0)}  \big\langle A^t (D^2 u (x,t) \pm\eta I)Ay, y \big\rangle
\,dy.  
\end{split}
\]
By Lemma~\ref{lemma.trace.integral} we have that
$$
\dashint_{B_{\varepsilon}(0)}  \big\langle A^t (D^2 u (x,t) \pm \eta I) Ay, y \big\rangle
\,dy   = 
\frac{\varepsilon^2}{n+2}\,\textrm{trace}\left(A^t\left(D^2u(x,t) \pm \eta I \right)A\right)
$$
and we obtain
\[
\begin{split} 
\mathop{\mathop{\inf}_{\det (A) \times b =1}}_{A\leq \phi(\varepsilon)I,\ b \leq \phi (\varepsilon)}&
\bigg\{
\dashint_{t-  \frac{b^2}{n+2}\,\varepsilon^2}^t\dashint_{B_{\varepsilon}(0)}
P_\eta^\pm (x+Ay, s)
\,dy\,ds
\bigg\} - u (x,t)  \\[10pt]
& =  
\frac{\varepsilon^2}{2(n+2)} 
\mathop{\mathop{\inf}_{\det (A) \times b =1}}_{A\leq \phi(\varepsilon)I,\ b \leq \phi (\varepsilon)} 
\Bigg\{ 
b^2 \left(-\frac{\partial u}{\partial t} (x,t) \pm \eta\right)   
 + \textrm{trace}\left(A^t\left(D^2u(x,t) \pm \eta I \right)A\right)  \Bigg\} .
 \end{split}
 \]

Observe that  $A\leq \phi( \varepsilon)I$ and $|y|\leq \varepsilon$ imply
 $x+Ay\in B_\delta(x)$  for $\varepsilon< \varepsilon_0$ (since $\eps\phi(\eps)\to 0$ as $\eps\to 0$). 
 Similarly, $b \leq \phi (\varepsilon)$ implies $\big(t-  \frac{b^2}{n+2}\,\varepsilon^2,t\big)\subset(t-\delta,t+\delta)$ for $\varepsilon< \varepsilon_0$.
Therefore,  \eqref{u-P.o.pequena.2} implies that
given $\eta>0$ there is $\varepsilon_0$ such that if $\varepsilon<\varepsilon_0$, then
 \begin{equation}\label{C2.case.second.part.88}
P_\eta^-(x+Ay,s)
\leq
u(x+Ay,s)
\leq
P_\eta^+(x+Ay,s)\quad\textrm{for every}\ (y,s)\in B_\varepsilon(0)\times\left(t-  \frac{b^2}{n+2}\,\varepsilon^2,t\right).
\end{equation}
Consequently,  we obtain
\begin{equation}\label{C2.case.final.loop.first.part.77}
\begin{split}
\frac{\varepsilon^2}{2(n+2)} &
\mathop{\mathop{\inf}_{\det (A) \times b =1}}_{A\leq \phi(\varepsilon)I,\ b \leq \phi (\varepsilon)} 
\Bigg\{ 
b^2 \left(-\frac{\partial u}{\partial t} (x,t) - \eta\right)   
 + \textrm{trace}\left(A^t\left(D^2u(x,t) - \eta I \right)A\right)  \Bigg\} 
\\
&\leq
\mathop{\mathop{\inf}_{\det (A) \times b =1}}_{A\leq \phi(\varepsilon)I,\ b \leq \phi (\varepsilon)}
\dashint_{t- \frac{b^2}{n+2}\varepsilon^2}^t \dashint_{B_{\varepsilon}(0)}
\left(
u(x+Ay,s)-u(x,t)
\right)
\,dy\,ds
\\
&\leq
\frac{\varepsilon^2}{2(n+2)} 
\mathop{\mathop{\inf}_{\det (A) \times b =1}}_{A\leq \phi(\varepsilon)I,\ b \leq \phi (\varepsilon)} 
\Bigg\{ 
b^2 \left(-\frac{\partial u}{\partial t} (x,t) + \eta\right)   
 + \textrm{trace}\left(A^t\left(D^2u(x,t) + \eta I \right)A\right)  \Bigg\} .
\end{split}
\end{equation}
Now, we   notice that 
\[
\begin{split}
b^2 \left(-\frac{\partial u}{\partial t} (x,t) \pm \eta\right)  
& + \textrm{trace}\left(A^t\left(D^2u(x,t)\pm\eta I \right)A\right)   
\\
&=
\textrm{trace}\left( C_{b,A}^t  \left[ \begin{array}{cc}
\displaystyle -\frac{\partial u}{\partial t} (x,t) \pm \eta   & 0  \\
0 &   D^2u(x,t) \pm\eta I
\end{array} \right] C_{b,A} \right),
\end{split}
\]
where the matrix $C_{b,A} \in \mathbb{S}^{n+1}$ is given by
$$
C_{b,A} = \left[ \begin{array}{cc}
b & 0  \\
0 &   A 
\end{array} \right].
$$
Then, from Lemma  \ref{lemma.convergence.eta}
(notice that  we assumed that $D^2u (x,t)$ is strictly positive and $- \displaystyle \frac{\partial u}{\partial t} (x,t)$ is also strictly positive, 
in this case we have $f(x,t) \neq 0$) we get that
\[
\begin{split}
\lim_{\eta\to 0}&\left\{
\mathop{\mathop{\inf}_{ \det (A)\times b=1}}_{A\leq \phi(\varepsilon)I,\ b \leq \phi (\varepsilon)}
\textrm{trace}\left( C_{b,A}^t  \left[ \begin{array}{cc}
\displaystyle - \frac{\partial u}{\partial t} (x,t)\pm\eta  & 0  \\
0 &  \displaystyle  D^2u(x,t) \pm\eta I
\end{array} \right] C_{b,A} \right) 
\right\}
\\
& = \lim_{\eta\to 0}\left\{ \inf_{ \det (A)\times b=1} 
\textrm{trace}\left( C_{b,A}^t  \left[ \begin{array}{cc}
\displaystyle - \frac{\partial u}{\partial t} (x,t) \pm\eta & 0  \\
0 & \displaystyle   D^2u(x,t) \pm\eta I 
\end{array} \right] C_{b,A} \right)  \right\} \\
 &= 
(n+1) \left( - \frac{\partial u}{\partial t} (x,t)\cdot \det\big(D^2u(x,t)  \big) \right)^{\frac{1}{n+1}} 
\end{split}
\]
for $\varepsilon$ small enough, and the result follows.

Now, if $$- \frac{\partial u}{\partial t} (x,t)\cdot\det(D^2u(x,t))=0$$ (i.e., $f(x,t)=0$), the parabolic convexity of $u$ and a minor modification of the previous arguments
yield 
\[
\mathop{\mathop{\inf}_{\det (A) \times b =1}}_{A\leq \phi(\varepsilon)I,\ b \leq \phi (\varepsilon)}
\bigg\{
\dashint_{t-  \frac{b^2}{n+2} \varepsilon^2}^t\dashint_{B_{\varepsilon}(0)}
u(x+Ay, b s)
\,dy\,ds
\bigg\} - u(x,t) 
\geq 
 o(\varepsilon^2).
\]
Since the upper bound is still at our disposal, the result follows.

Finally, for the viscosity version of our asymptotic mean value formulas, we proceed as in the proof of Theorem \ref{te1.intro}, by
touching $u$ from above or below with a parabolically convex, smooth test function $\phi$
and using our previous computations that showed that a smooth parabolically convex
function verifies an inequality like 
 \begin{equation} \label{monge.ampere.2.88}
 - \frac{\partial \phi}{\partial t} (x,t) \cdot \mbox{det} (D^2\phi (x,t))  \geq  f(x,t).
 \end{equation}
if and only if  
\begin{equation}\label{MVP.pabob.2.88}
\begin{split}
\phi(x,t)
\leq 
\mathop{\mathop{\inf}_{\det (A) \times b =1}}_{A\leq \phi(\varepsilon)I,\ b \leq \phi (\varepsilon)}
\bigg\{
\dashint_{t- \frac{b^2}{n+2} \varepsilon^2}^t\dashint_{B_{\varepsilon}(0)}
\phi (x+Ay, s)
\,dy\,ds
\bigg\}
-\frac{n+1}{2(n+2)}\,(f(x,t))^{\frac{1}{n+1}}\,\varepsilon^2
+o(\varepsilon^2)
\end{split}
\end{equation}
proving our characterization of viscosity solutions to \eqref{monge.ampere.2.999}. 
\end{proof}

\begin{remark}
With the same ideas, 
from the formula
$$
{\rm trace}(M)
=
\frac{n}{\varepsilon^2}
\dashint_{\partial B_\varepsilon (0)}
\langle My,y\rangle
\,d\mathcal{H}^{n-1}(y),
$$
we can obtain an asymptotic mean value formula involving means over surfaces of ellipsoids:
$u$ is a
viscosity solution to  the parabolic Monge-Amp\`ere equation 
\begin{equation} \label{monge.ampere.1.s.77.99.567}
- \frac{\partial u}{\partial t} (x,t)\cdot  \det{\big(D^2u (x,t)\big)} = f(x,t).
 \end{equation}
if and only if 
\begin{equation}\label{MVP.pabob.1.999.456.88}
\begin{split}
u(x,t)
=
\mathop{\mathop{\inf}_{\det (A) \times b =1}}_{A\leq \phi(\varepsilon)I,\ b \leq \phi (\varepsilon)}
\bigg\{
\dashint_{t- \frac{b^2}{n} \varepsilon^2}^t
\dashint_{\partial B_{\varepsilon}(0)}
u(x+Ay,s)
\,dy\,ds
\bigg\}
+\frac{n+1}{2 n}\, (f(x,t))^{\frac{1}{n+1}}\,\varepsilon^2
+o(\varepsilon^2)
\end{split}
\end{equation}
as $\varepsilon\to0$
in the viscosity sense. 
\end{remark}

\section{Infimum and inf-sup operators}\label{sec.infimums}

\subsection{Infimum operators}
In this section we deal with equations of the form
\begin{equation}\label{operator.bounded.A.88}
F\Big(x,t,\frac{\partial u}{\partial t}, D^2u \Big)=\inf_{(A,b)\in\A_{x,t}} \Big\{ \tr(A^tD^2u(x,t)A)- b \frac{\partial u}{\partial t} (x,t) \Big\} =0 .
\end{equation}
Here $\A_{x,t}\subset \mathbb{S}_+^n(\mathbb{R})\times \mathbb{R}^+$  
are bounded subsets for each $x\in\R$ and $t\in [t_1,t_2]$.

Observe that we are assuming that the sets $\A_{x,t}$ are bounded. 
This has to be contrasted to the previous case, the Monge-Amp\`ere case, where
the set of relevant matrices $\mathcal{A}_x=\{A\in \mathbb{S}_+^n(\mathbb{R}): \det(A)=~1\}$ is unbounded.

The fact that we are taking a bounded set of coefficients in \eqref{operator.bounded.A.88} is equivalent to
assuming that $F$ is well defined and finite.

\begin{lemma}
\label{Axbounded} The operator
$$
F(x,t,z,M)= \inf_{(A,b)\in\A_{x,t}} \Big\{ \tr(A^t M A)- b z \Big\} 
$$
is finite for every $M\in  \mathbb{S}^n(\R)$ and every $z \in \mathbb{R}$ if and only if
$\A_{x,t}$ is bounded.
\end{lemma}

\begin{proof} Suppose that $\A_{x,t}$ is not bounded. First, assume that
there exists a sequence of matrices $A_k\in \A_x$ such that their largest eigenvalue $\lambda_k$ diverges.
Let $v_k$ be the corresponding unitary eigenvectors.
Since $v_k$ are unitary vectors, we can extract a subsequence (denoted equal) that has a limit, that is, $v_k\to v$.
Let $M$ be the symmetric matrix with eigenvector $v$ of eigenvalue $-1$ and eigenvalue 0 with multiplicity $n-1$.
Then
\[
\tr (A_k^tMA_k)\leq -\lambda_k^2 \|\textrm{proj}_{v}(v_k)\|+\lambda_k^2 \|\textrm{proj}_{v^\bot}(v_k)\|.
\]
Which is a contradiction since the RHS goes to $-\infty$.

Next, if the set $\A_{x,t}$ contains an unbounded set of $b$. If $b_k \to + \infty$, we just take $z=-1$ and if $b_k \to -\infty$ then take $z=1$, to obtain
the contradiction.  
\end{proof}

Now we are ready to prove Theorem \ref{th.sol.viscosa.intro}.

\begin{proof}[Proof of Theorem \ref{th.sol.viscosa.intro}]
As in the previous proofs, let us use the Taylor expansion of $u(y,s)$,
$$
\begin{array}{l}
\displaystyle
u(y,s) = u(x,t) + \frac{\partial u}{\partial t} (x,t) (s-t)  +
\langle \nabla u(x,t), (y-x) \rangle  \\[10pt]
\qquad \qquad  \displaystyle + \frac12 \langle D^2 u (x,t) (y-x), (y-x) \rangle + o(|t-s|+|x-y|^2).
\end{array}
$$

Notice that now, since the set of involved coefficients (matrices) are bounded then the error terms
in the Taylor expansion
are uniform, therefore we can avoid the use of paraboloids here. 

As before, we start with the proof of the case in which the solution $u$ is a classical $C^{2,1}$ solution.
We have
$$
\begin{array}{l}
\displaystyle 
\dashint_{t-  \frac{b}{n+2} \varepsilon^2}^t\dashint_{B_{\varepsilon}(0)}
u (x+Ay, s)
\,dy\,ds
 - u (x,t) \\[10pt]
\displaystyle = \dashint_{t- \frac{b}{n+2} \varepsilon^2}^t\dashint_{B_{\varepsilon}(0)} \frac{\partial u}{\partial t} (x,t) (s-t)  \,dy\,ds
\\[12pt]
\qquad \displaystyle + \dashint_{t- \frac{b}{n+2} \varepsilon^2}^t\dashint_{B_{\varepsilon}(0)} \langle \nabla u(x,t), Ay \rangle + 
\frac12 \langle D^2 u (x,t) Ay, Ay \rangle \,dy\,ds + o(\varepsilon^2) 
\\[12pt]
\displaystyle = \frac{\partial u}{\partial t} (x,t) \dashint_{t- \frac{b}{n+2}\varepsilon^2}^t (s-t)  \,ds
  + \dashint_{B_{\varepsilon}(0)} \frac12 \langle A^t (D^2 u (x,t) ) Ay, y \rangle
\,dy + o(\varepsilon^2) 
\\[12pt]
\displaystyle = -   \frac{b}{2(n+2)} \varepsilon^2 \, \frac{\partial u}{\partial t} (x,t) 
 +\dashint_{B_{\varepsilon}(0)} \frac12 \langle A^t (D^2 u (x,t) ) Ay, y \rangle
\,dy  + o(\varepsilon^2) .
\end{array}
$$

Now, we use one more time Lemma \ref{lemma.trace.integral} to obtain
$$
\dashint_{B_{\varepsilon}(0)} \frac12 \langle A^t (D^2 u (x,t) ) Ay, y \rangle
\,dy = \frac{\varepsilon^{2}}{n+2}\,
{\rm trace}(A^t (D^2 u (x,t) ) A).
$$

Therefore, we have 
\[
\begin{split}
\dashint_{t-  \frac{b}{n+2} \varepsilon^2}^t&\dashint_{B_{\varepsilon}(0)}
u (x+Ay, s)
\,dy\,ds
 - u (x,t) \\
& = \frac{\varepsilon^2}{2(n+2)} \left ( - b  \frac{\partial u}{\partial t} (x,t) 
 + 
{\rm trace}(A^t (D^2 u (x,t) ) A)  \right) + o(\varepsilon^2) .
\end{split}
\]
Hence, we conclude that
$$
\begin{array}{l}
\displaystyle
\inf_{(A,b)\in\A_{x,t}}
\dashint_{t- \frac{b}{n+2} \varepsilon^2}^t \dashint_{B_{\varepsilon}(0)}
u(x+Ay,s)
\,dy\,ds
-
u(x,t) 
\\[12pt]
\displaystyle =
 \frac{\varepsilon^2}{2(n+2)} \inf_{(A,b)\in\A_{x,t}} \left ( - b  \frac{\partial u}{\partial t} (x,t) 
 + 
{\rm trace}(A^t (D^2 u (x,t) ) A)  \right) + o(\varepsilon^2) 
\\[12pt]
\displaystyle 
=  \frac{\varepsilon^2}{2(n+2)} F\Big(x,t,\frac{\partial u}{\partial t}, D^2u \Big) +
o(\varepsilon^2) \\[12pt]
= o(\varepsilon^2)
\end{array}
$$
proving the asymptotic mean value formula in a pointwise sense for smooth solutions. 

As a consequence of the mean value property for smooth functions we can obtain the 
characterization for viscosity solutions in Theorem \ref{th.sol.viscosa.intro}. The details are left to the reader.
\end{proof}

\begin{remark} \label{rem-sub-super}
The previous proof shows that in fact viscosity supersolutions (subsolutions) to
\begin{equation}\label{operator.bounded.A.88.999}
\inf_{(A,b)\in\A_{x,t}} \Big\{ \tr(A^tD^2u(x,t)A)- b \frac{\partial u}{\partial t} (x,t) \Big\} =0 .
\end{equation}
are characterized by 
$$
\inf_{(A,b)\in\A_{x,t}}
\dashint_{t- \frac{b}{n+2} \varepsilon^2}^t \dashint_{B_{\varepsilon}(0)}
u(x+Ay, s)
\,dy\,ds
-
u(x,t) \leq o(\varepsilon^2) \quad (\geq o(\varepsilon^2)),
$$
as $\varepsilon \to 0$ in the viscosity sense. 
\end{remark}

{\bf Examples.} Let us mention some examples of operators and the corresponding set of matrices $\A$ such that
the previous results apply.
To that end we denote $\lambda_1(M)\leq \lambda_2(M)\leq \cdots \leq \lambda_n(M)$ the eigenvalues of the matrix $M$, 
or simply $\lambda_i$ when the matrix involved is clear from the context.

{\bf Parabolic equations related to Pucci operators.} An important example under these assumptions are the parabolic Pucci operators. 
For given $0<\theta<\Theta$, 
consider the second order, uniformly parabolic operators,
\[
\mathcal{M}^-_{\theta,\Theta} \Big(\frac{\partial u}{\partial t}, D^2u \Big)=
\frac{\partial u}{\partial t} - \left( \theta \sum_{\lambda_i(D^2u)>0}\lambda_i(D^2u)+\Theta\sum_{\lambda_i(D^2u)<0}\lambda_i(D^2u) \right)
\]
and
\[
\mathcal{M}^+_{\theta,\Theta} \Big(\frac{\partial u}{\partial t}, D^2u \Big)= \frac{\partial u}{\partial t} - 
\left( \Theta\sum_{\lambda_i(D^2u)>0}\lambda_i(D^2u)+\theta\sum_{\lambda_i(D^2u)<0}\lambda_i(D^2u) \right).
\]

 In this case, we have as the involved set of matrices 
 \begin{equation}\label{A.theta.Theta}
 \mathcal{A}=\left\{A\in \mathbb{S}_+^n(\mathbb{R}): \sqrt{\theta} \leq \lambda_i(A)\leq \sqrt{\Theta}\right\}
\end{equation}
that is bounded uniformly in $x$. In fact, we can write
$$
\left( \theta \sum_{\lambda_i(M)>0}\lambda_i(M)+\Theta\sum_{\lambda_i(M)<0}\lambda_i(M) \right)=\inf_{A\in\A} \tr(A^tMA)
$$
and
$$
\left( \theta \sum_{\lambda_i(M)>0}\lambda_i(M)+\Theta\sum_{\lambda_i(M)<0}\lambda_i(M) \right)=\inf_{A\in\A} \tr(A^tMA)
=\sup_{A\in\A} \tr(A^tMA).
$$

For these operators we have: $u$ is a viscosity solution to
$$
\frac{\partial u}{\partial t} (x,t)- \left( \theta \sum_{\lambda_i(D^2u (x,t))>0}\lambda_i(D^2u (x,t))+\Theta\sum_{\lambda_i(D^2u(x,t))<0}
\lambda_i(D^2u(x,t)) \right)
=0
$$
if and only if 
$$
u(x,t) = \inf_{\sqrt{\theta} \leq \lambda_i(A)\leq \sqrt{\Theta}}  
\dashint_{t- \frac{\varepsilon^2}{n+2} }^t \dashint_{B_{\varepsilon}(0)}
u(x+Ay, s)
\,dy\,ds
 + o(\varepsilon^2)
$$
as $\varepsilon \to 0$ in the viscosity sense.

{\bf The evolution problem associated with the convex envelope, $\lambda_1(D^2u)$.}
Notice, however, that the operators that satisfy \eqref{operator.bounded.A} do not need to be uniformly elliptic. For example, the operator 
\[
\frac{\partial u}{\partial t} - \lambda_1(D^2u),
\]
 studied in \cite{BER} in connection with the convex envelope of a function, corresponds to the set of matrices 
 $$\mathcal{A}=\Big\{A\in \mathbb{S}_+^n(\mathbb{R}): \lambda_1=\cdots=\lambda_{n-1}=0 \text{ and } \lambda_n=1\Big\}.$$

In this case the asymptotic mean value formula reads as 
$$
u(x,t) = \inf_{|v|=1}
\dashint_{t- \frac{\varepsilon^2}{n+2} }^t \dashint_{B_{\varepsilon}(0)}
u\big(x+ \langle y , v \rangle v, s\big)
\,dy\,ds
 + o(\varepsilon^2).
$$

\subsection{Sup-Inf operators}\label{sec.sup-inf}

We also consider a special type of Isaacs operators where for every $x$ the supremum (or the infimum) is taken over a subset $\mathbb{A}_{x,t}$ of the parts of $S_+^n(\mathbb{R})\times\mathbb{R}^+$. More precisely, let $\mathbb{A}_{x,t} \subset \mathcal P (S_+^n(\mathbb{R})\times\mathbb{R}^+)$ be a subset for each $x\in\R^n$ and $t>0$  such that 
\[
\bigcup\mathbb{A}_{x,t}
=
\Big\{
(A,b)\in \mathbb{S}_+^n(\mathbb{R})\times\mathbb{R}^+\ : \
(A,b)\in \mathcal{A}\ \textrm{for some}\ \mathcal{A}\in\mathbb{A}_{x,t}
\Big\}
\]
is bounded. \normalcolor
 
 We consider the differential operator $F:\Omega \times \R \times \R \times S^n(\R)\to \R$ given by
\begin{equation}
F\Big(x, t, \frac{\partial u}{\partial t}, D^2u \Big) =\sup_{\mathcal{A}\in\mathbb{A}_{x,t}} \inf_{(A,b)\in \mathcal{A}} 
\Big\{ \textrm{trace}(A^tD^2u(x,t)A)
- b \frac{\partial u}{\partial t} (x,t) \Big\}.
\end{equation}

Since $\cup\mathbb{A}_{x,t}$ is bounded, Theorem \ref{th.sol.viscosa.intro.22} follows as before,
and we have the asymptotic mean value characterization of  viscosity solutions of
\begin{equation}
F\Big(x, t, \frac{\partial u}{\partial t}, D^2u \Big)  =0, 
\end{equation} 
as those continuous functions that satisfy
\begin{equation}
\sup_{\mathcal{A}\in\mathbb{A}_{x,t}} \inf_{(A,b)\in \mathcal{A}}
\dashint_{t- \frac{b}{n+2} \varepsilon^2}^t \dashint_{B_{\varepsilon}(0)}
u(x+Ay, s)
\,dy\, ds
-
u(x,t)
=
o(\varepsilon^2),
\end{equation}
as $\eps\to0$ in the sense of viscosity.

We leave the
details to the reader.

{\bf Examples.} As examples of operators such that
the previous results apply we mention the following:

{\bf Evolution problems associated with eigenvalues of the Hessian.} 
This allows us to prove asymptotic mean value formulas for degenerate parabolic operators such as 
$$
\frac{\partial u}{\partial t} = \lambda_k (D^2u)
$$
that where studied in \cite{BER}. Here $\lambda_k (D^2u)$ stands for
the $k$-th smallest eigenvalue of the Hessian, given by the  Courant--Fischer min-max principle
\begin{equation} \label{Courant-Fisher}
\lambda_{k}\big(D^2u(x)\big)= \min_{\textrm{dim}(V) = n-k+1} \ 
\mathop{\mathop{\min}_{v \in V}}_{|v|=1}   \Big\{  \langle D^2u(x)v,v\rangle  \Big\}.
\end{equation}
For example, for the operator $\lambda_2(D^2u)$ we have the set of sets of matrices
\[
\A=\Big\{\{A\in {\mathbb{S}}_+^n(\mathbb{R}): \lambda_1=\cdots=\lambda_{n-1}=0, \lambda_n=1 \text{ and } v_n\in V\}: 
V \text{ a subspace of dimension } n-1 \Big\}.
\]

In this case the asymptotic mean value formula is given by 
$$
u(x,t) = \sup_{\textrm{dim}(V)= n-k+1}  \
\mathop{\mathop{\inf}_{v \in V}}_{|v|=1} 
\dashint_{t- \frac{\varepsilon^2}{n+2} }^t \dashint_{B_{\varepsilon}(0)}
u\big(x+ \langle y , v \rangle v, s\big)
\,dy\,ds
 + o(\varepsilon^2).
$$

For example, for the equation
$$
\frac{\partial u}{\partial t} (x,t) - \lambda_2(D^2u (x,t)) = 0,
$$
we can write the following asymptotic mean value formula
$$
u(x,t) = 
\sup_{|w|=1}
\mathop{\mathop{\inf}_{\langle v,w \rangle =0}}_{|v|=1} 
\dashint_{t- \frac{\varepsilon^2}{n+2}}^t \dashint_{B_{\varepsilon}(0)}
u\big(x+ \langle y , v \rangle v, s\big)
\,dy\,ds
 + o(\varepsilon^2).
$$

\section{A probabilistic interpretation.} \label{sec-DPP}

The asymptotic mean value formulas that we obtained can be interpreted in terms of  Dynamic Programming Principles for two-player zero-sum games. 

First, we describe a closely related random walk.
For a small $\eps>0$ and a fixed matrix $A$ consider the following random walk in a bounded domain $\Omega_T=\Omega \times (0,T)$. 
From $(x,t)$ the next spacial position is given by $x+Ay$ with $y\in B_{\varepsilon}$ chosen with uniform distribution in the ball and the new time is $t -  \frac{b}{2(n+2)} \varepsilon^2$.
The process continues until the spacial position leaves the $\Omega$ or when the time becomes nonpositive. 
We call $\tau$ the stopping time given by the number of plays until the position of the game leaves $\Omega_T$ and call $(x_\tau,t_\tau)$ the last position of the process. 
We have $(x_\tau,t_\tau)\not\in \Omega_T$, even more, $(x_\tau,t_\tau)\in \Omega_T^c:= \R^n\times (-\infty,T) \setminus \Omega_T$.

We fix a final payoff functions $g:\Omega_T^c\to \R$ and we define
$$
v^\varepsilon (x,t) = \mathbb{E}^{x,t} [g(x_\tau,t_\tau)].
$$
Then, it follows that $v^\varepsilon$ verifies 
$$
v^\varepsilon (x,t)
=
\dashint_{B_{\varepsilon}(0)}
v^\varepsilon \Big(x+Ay, t - \frac{b}{2(n+2)} \varepsilon^2 \Big)
\,dy
$$
for $x\in \Omega$, $t>0$.
That is, the expected final payoff starting at $(x,t)$ is equal to the average of the expectation over all the possible next positions.

Next we describe the two-player zero-sum game.
We are given $\mathbb{A}_{x,t}\subset \mathcal P (S_+^n(\mathbb{R})\times (0,+\infty))$ for each $(x,t)\in \Omega_T$.
The game starts at $(x_0,t_0) \in \Omega_T$, 
the first player, Player~I, chooses a set of matrices and scalars
$\mathcal{A}\in\mathbb{A}_{x,t}$ (she chooses a casino in probabilistic terms) and next the second player, Player~II, chooses a 
a matrix and a scalar $(A,b) \in \mathcal{A}$
(she chooses a game to play in the casino chosen by Player~I).
The next position of the game is given by the previously described random walk, that is, $$x_1 = x_0 + Ay$$ with $y\in B_{\varepsilon}$ being 
chosen according to the uniform distribution
in $B_{\varepsilon}$ and 
$$
t_1=t_0 - b \frac{1}{2(n+2)} \varepsilon^2.
$$
The game continues from $x_1$ accordingly to the same rules.
The game finishes at the first time at which the position $(x_\tau,t_\tau)$ leaves $\Omega_T$.
At this point, Player~I gets $g(x_\tau, t_\tau)$ and Player~II gets $-g(x_\tau, t_\tau)$ (one can think that Player~II pays to Player~I the amount given by 
the final payoff).

When the two 
players decide what to play at each turn (they choose theirs strategies), we can compute expected payoff (the expected earnings for Player~I)
playing with strategies $S_I$ and $S_{II}$ starting at $(x,t) \in \Omega_T$ as
$$
\mathbb{E}_{S_I,S_{II}}^{x,t} [g(x_\tau, t_\tau)].
$$

Then, the extreme values for this game are given by 
$$
u^I (x,t) =  \inf_{S_{II}}\sup_{S_I} \mathbb{E}_{S_I,S_{II}}^{x,t} [g(x_\tau, t_\tau)]
$$
and
$$
u^{II} (x,t) =  \sup_{S_I} \inf_{S_{II}} \mathbb{E}_{S_I,S_{II}}^{x,t} [g(x_\tau, t_\tau)].
$$
We are taking $\sup$ and $\inf$ over the strategies for Player~I and Player~II respectively since Player~I wants to maximize the value of $g(x_\tau, t_\tau)$ while Player~II is trying to minimize it.
When these two extreme values coincide, we say that the game has a value given by
$$
u^\varepsilon (x,t) := u^{II} (x,t) = u^{I} (x,t).
$$
It turns out that $u^\varepsilon (x,t)$ is a solution to the Dynamic Programming Principle associated to this game
that reads as
$$
u^\varepsilon (x,t) = \sup_{\mathcal{A}\in\mathbb{A}_{x,t}} \inf_{A\in \mathcal{A}}
\dashint_{B_{\varepsilon}(0)}
u^\varepsilon\left(x+Ay, t - b \frac{1}{2(n+2)} \varepsilon^2\right)
\,dy.
$$
Notice that this formula is just our asymptotic mean value formula \eqref{pepe} without the error term. 

When the next time position is also chosen with uniform probability in the interval $(t - b \frac{1}{n+1} \varepsilon^2,t)$
we obtain 
$$
u^\varepsilon (x,t)
=
\sup_{\mathcal{A}\in\mathbb{A}_x} \inf_{A\in \mathcal{A}} \dashint_{t - \frac{b}{n+2} \varepsilon^2}^t
\dashint_{B_{\varepsilon}(0)}
u^\varepsilon (x+Ay, s)
\,dy \, ds,
$$
where the reader can recognize the left-hand side of \eqref{sup.inf.MVP.formula} without the error term. 

For the mean value formulas associated with Monge-Amp\`ere operators 
we just consider a game with only one player (a controller) that chooses 
at each point the parameters involved in the random walk (the controller chooses the matrix $A$
involved in the random walk). We add a running payoff (at each point the player pays $\frac{n}{2(n+2)} f(x,t) \varepsilon^2$).
In this case the value of the game is given by
$$
u^\varepsilon  (x,t) =  \inf_{S_I} \mathbb{E}_{S_I}^{x,t} \Big[g(x_\tau, t_\tau) + \sum_{j=0}^{\tau-1} \frac{n}{2(n+2)} f(x_j,t_j) \varepsilon^2 \Big],
$$
and the associated Dynamic Programming Principle reads as
$$
u^\varepsilon (x,t)
=
\mathop{\mathop{\inf}_{\det (A)=1}}_{A\leq \phi(\varepsilon)I}
\bigg\{
\dashint_{t- \frac{n}{n+2} \varepsilon^2}^t\dashint_{B_{\varepsilon}(0)}
u^\varepsilon (x+Ay,s)
\,dy\,ds
\bigg\}
+ \frac{n}{2(n+2)}\, f(x,t)\,\varepsilon^2.
$$

When the controller also chooses  
the coefficient $b$ involved in the time step and the running payoff is given by
$- \frac{n+1}{2(n+2)}\,(f(x,t))^{\frac{1}{n+1}} \varepsilon^2$, we get 
$$
u^\varepsilon(x,t)
=
\mathop{\mathop{\inf}_{\det (A) \times b =1}}_{A\leq \phi(\varepsilon)I,\ b \leq \phi (\varepsilon)}
\bigg\{
\dashint_{t-  \frac{b^2}{n+2} \varepsilon^2}^t\dashint_{B_{\varepsilon}(0)}
u^\varepsilon (x+Ay, s)
\,dy\,ds
\bigg\}
-\frac{n+1}{2(n+2)}\,(f(x,t))^{\frac{1}{n+1}}\,\varepsilon^2.
$$
For references on this program involving games and PDEs we refer to \cite{BR,[Manfredi2012], PSSW}, the book \cite{[Blanc and Rossi 2019]}, and references therein.

\medskip


\bibliographystyle{plain}

\end{document}